\documentclass[reqno,11pt]{amsart}
\usepackage{amsmath, latexsym, amsfonts, amssymb, amsthm, amscd}
\usepackage{mathrsfs}
\usepackage{graphics,epsf,psfrag}
\usepackage[lofdepth,lotdepth]{subfig}
\usepackage{graphicx}
\usepackage{url}
\usepackage[dvipsnames]{xcolor}

\usepackage{hyperref}

\numberwithin{equation}{section}

\newtheorem{theorem}{Theorem}[section]

\newtheorem{proposition}[theorem]{Proposition}
\newtheorem{corollary}[theorem]{Corollary}
\newtheorem{remark}[theorem]{Remark}

\newcommand{\cF}{{\ensuremath{\mathcal F}} }

\newcommand{\cN}{{\ensuremath{\mathcal N}} }
\newcommand{\cL}{{\ensuremath{\mathcal L}} }

\newcommand{\bE}{{\ensuremath{\mathbf E}} }
\newcommand{\bP}{{\ensuremath{\mathbf P}} }

\DeclareMathSymbol{\leqslant}{\mathalpha}{AMSa}{"36} 
\DeclareMathSymbol{\geqslant}{\mathalpha}{AMSa}{"3E} 
\DeclareMathSymbol{\eset}{\mathalpha}{AMSb}{"3F}     
\renewcommand{\leq}{\;\leqslant\;}                   
\renewcommand{\geq}{\;\geqslant\;}                   
\newcommand{\dd}{\,\text{\rm d}}             

\newcommand{\sumtwo}[2]{\sum_{\substack{#1 \\ #2}}} 

\newcommand{\bbE}{{\ensuremath{\mathbb E}} }

\newcommand{\bbN}{{\ensuremath{\mathbb N}} }

\newcommand{\bbP}{{\ensuremath{\mathbb P}} }

\newcommand{\bbR}{{\ensuremath{\mathbb R}} }
\newcommand{\bbS}{{\ensuremath{\mathbb S}} }

\newcommand{\bbZ}{{\ensuremath{\mathbb Z}} }

\newcommand{\ga}{\alpha} 
 
\newcommand{\gd}{\delta} 
\newcommand{\gep}{\varepsilon}       

\newcommand{\gG}{\Gamma}

\newcommand{\go}{\omega} 
 
\newcommand{\gl}{\lambda} 
 
\newcommand{\gs}{\sigma}

\makeatletter
\def\captionfont@{\footnotesize}
\def\captionheadfont@{\scshape}

\long\def\@makecaption#1#2{%
  \vspace{2mm}
  \setbox\@tempboxa\vbox{\color@setgroup
    \advance\hsize-6pc\noindent
    \captionfont@\captionheadfont@#1\@xp\@ifnotempty\@xp
        {\@cdr#2\@nil}{.\captionfont@\upshape\enspace#2}%
    \unskip\kern-6pc\par
    \global\setbox\@ne\lastbox\color@endgroup}%
  \ifhbox\@ne 
    \setbox\@ne\hbox{\unhbox\@ne\unskip\unskip\unpenalty\unkern}%
  \fi
  \ifdim\wd\@tempboxa=\z@ 
    \setbox\@ne\hbox to\columnwidth{\hss\kern-6pc\box\@ne\hss}%
  \else 
    \setbox\@ne\vbox{\unvbox\@tempboxa\parskip\z@skip
        \noindent\unhbox\@ne\advance\hsize-6pc\par}%
\fi
  \ifnum\@tempcnta<64 
    \addvspace\abovecaptionskip
    \moveright 3pc\box\@ne
  \else 
    \moveright 3pc\box\@ne
    \nobreak
    \vskip\belowcaptionskip
  \fi
\relax
}
\makeatother
\def\writefig#1 #2 #3 {\rlap{\kern #1 truecm
\raise #2 truecm \hbox{#3}}}

\makeatletter
\let\orgdescriptionlabel\descriptionlabel
\renewcommand*{\descriptionlabel}[1]{%
  \let\orglabel\label
  \let\label\@gobble
  \phantomsection
  \edef\@currentlabel{#1}%
  \let\label\orglabel
  \orgdescriptionlabel{#1}%
}
\makeatother

\begin{document}

\title[Random graphs and Fokker-Planck equations]{A note on dynamical models on random graphs and 
Fokker-Planck 
equations}

\author{Sylvain Delattre}
\address{Universit\'e Paris Diderot, Sorbonne Paris Cit\'e,  Laboratoire de Probabilit{\'e}s et Mod\`eles Al\'eatoires, UMR 7599,
            F-75205 Paris, France}

\author{Giambattista Giacomin}
\address{Universit\'e Paris Diderot, Sorbonne Paris Cit\'e,  Laboratoire de Probabilit{\'e}s et Mod\`eles Al\'eatoires, UMR 7599,
            F-75205 Paris, France}

\author{Eric Lu\c{c}on}
\address{Universit\'e Paris Descartes, Sorbonne Paris Cit\'e, Laboratoire MAP5, UMR 8145, 75270 Paris, France}

\begin{abstract}
We address the issue of the proximity of interacting diffusion models on large graphs with a uniform degree property and a corresponding mean field model, i.e. a model on the complete graph with a suitably renormalized interaction parameter.
Examples include Erd\H{o}s-R\'enyi graphs with edge probability $p_n$, $n$ is the number of vertices, such that $\lim_{n \to \infty}p_n n= \infty$. The purpose of this note is twofold: (1) to establish this proximity on finite time horizon, by exploiting the fact that both systems are accurately described by a Fokker-Planck PDE (or, equivalently, by a nonlinear diffusion process) in the $n=\infty$ limit;  (2)   to remark that in reality this result is  unsatisfactory when it comes to applying it to systems with $n$  large but finite, for example the values of $n$ that can be reached in simulations or that correspond to the typical number of interacting units in a biological system. 
 \\[10pt]
  \\[10pt]
\end{abstract}

\date{\today}

\maketitle


\section{Introduction, results, discussion}
\label{sec:meanfieldlimit}

\subsection{The model}
For $n=2,3, \ldots$, choose an {\sl adjacency} matrix $\xi=\{\xi_{ i,j}^{ (n)}\}_{ (i, j)\in\{1, \ldots, n\}^2}$, that is  $\xi_{ i,j}^{ (n)}\in\{0, 1\}$. Consider the graph $G^{ (n)}=(V^{(n)}, E^{ (n)})$ associated with this sequence, namely $V^{ (n)}:=\{1, \ldots, n\}$ and $E^{ (n)}= \{(i,j)\in V^{ (n)}\times V^{ (n)}:\,  \xi_{ i, j}^{ (n)}=1\}$. 
For $i\in V^{ (n)}$, define also the degree of the vertex $i$ as $d_{ i}^{ (n)}:=\sum_{j=1}^n \xi_{ i,j}^{ (n)}$.
Consider the $n$-dimensional diffusion 
\begin{equation}
\label{eq:odegene}
\qquad \dd\theta_{i, t}\, =\, F (\theta_{i, t}, \omega_{i})\dd t + \frac{\alpha_{ n}}{n}\sum_{j=1}^{ n} \xi_{ i, j}^{ (n)}\Gamma\left(\theta_{i, t}, \omega_{i}, \theta_{j, t}, \omega_{j} \right) \dd t + \sigma\dd B_{i, t}\, ,
\end{equation}
where $\{B_{i, \cdot}\}_{i=1,\ldots, n}$ is a sequence of standard IID Brownian motions with respect to a given filtration $\{\cF_t\}_{t\ge 0}$, $\{\go_i\}_{ i=1, \ldots, n}$ is a sequence of elements of a (complete and separable) metric space (seen as a disorder), $ \alpha_{ n}\ge 1$ is a positive number that we introduce to properly normalize the interaction, as it will be clear from the examples. In \eqref{eq:odegene}, $\gs \ge 0$ (for $\gs=0$, the diffusion is degenerate and \eqref{eq:odegene} is a  $n$-dimensional deterministic dynamical system).
We assume also that: 
\smallskip

\begin{itemize}
\item
$F:\bbR^2 \to \bbR$ is bounded and $F(\cdot, \go)$ is Lipschitz uniformly in $\go$.  
\item 
$\Gamma :\bbR^4 \to \bbR$ is bounded and $\Gamma(\cdot, \go, \cdot, \go')$ is Lipschitz uniformly in $\go$ and $\go'$. 
\end{itemize}

\smallskip

These conditions are sufficient to ensure uniqueness of a strong solution once we have fixed the initial conditions 
$\{\theta_{i, 0}\}_{i=1, \ldots, n}$ that are $n$ square integrable and $\cF_0$-measurable random variables. We denote by
$\bE[\, \cdot\,]$ the expectation with respect to the Brownian motions $\{B_{ i, \cdot}\}_{ i=1, 2,\ldots}$.  
We will 
 choose $\{( \theta_{j,0}, \go_j)\}_{j=1, 2, \ldots}$ to be (the realization of) a sequence of IID random vectors with common law denoted by $\nu_0$ and we  denote by
 $ \bbE\left[\cdot\right]$ the expectation w.r.t. to these variables. 
 Finally, also the graph itself will be random at times and
  $ \mathbb{ E}_{  \xi} \left(\cdot\right)$ is the corresponding expectation: the graph is independent of all other randomness in the system, that is disorder $\go$, initial condition and Brownians. A last remark is that the results are easily generalized to triangular arrays, that is assuming independence for 
  $\{(\theta_{j,0}, \go_j\}_{j=1, \ldots, n}$ with common law that may depend on $n$ and converges (weakly) to $\nu_0$ for $n \to \infty$: we will not do this to keep things simpler.
  
\subsection{A class of examples: Kuramoto model and variants}
The well known Kuramoto model \cite{cf:K}, even in various  generalized  versions, falls into the framework of \eqref{eq:odegene}. The original formulation is with   $G^{(n)}$ the complete graph and  $\ga_n=1$ for every $n$. Said otherwise, the original, or basic, Kuramoto model is just a mean field model. Of course a  more general choice of $G^{(n)}$ has an obvious meaning (and modeling relevance! In the Kuramoto context this issue has been considered from a numerical viewpoint in \cite{cf:torcini}, see also \cite[Sec.~IV.B]{cf:acebron}  and \cite{cf:VMP} for other types of graph disorder) and it is the point of this note, but let us stick to the classical framework  for now: 
\smallskip

\begin{itemize}
\item  If $\go$ is a real number, $F(\theta, \go)=\go$ and $\Gamma(\theta, \go, \theta',\go')=-K \sin(\theta-\theta')$, $K\ge 0$, then we are dealing with  the standard (stochastic) 
Kuramoto model \cite{cf:K,cf:acebron}. In this case we can view  the variables $\{\theta_{ i}\}_{ i=1,\ldots, n}$ as phases, that is we can map them to $\bbS :=\bbR/(2 \pi \bbZ)$, and the process
becomes a system of coupled diffusions on  $\bbS^n$. Since we require $F$ to be bounded, our results apply only to the case of bounded random variables $\go_i$.
\item $F(\theta, \go)=1-a \sin(\theta)$, $a \in \bbR$, and $\Gamma(\theta, \go, \theta',\go')=-K \sin(\theta-\theta')$, $K\ge 0$ is the 
active rotator model \cite{cf:KSS,cf:shinomoto1986a}. Note that in this case $\go$ does not enter, but one can  choose $a$ dependent on $\go$  and one would be dealing with a family of {\sl very} inhomogeneous interacting diffusion \cite{cf:GLP}. 
\item The generalization of the Kuramoto model presented for example in \cite{cf:VMP}
fits inside this framework with the choice
\begin{equation}
F(\theta, \go)= \eta\   \text{ and } \  \gG(\theta, \go, \theta',\go')=- A A^{ \prime} K \sin(\theta-\ga -\theta'+\ga'-\gd)\, ,
\end{equation}
where $\go=(\eta, A,  \ga )\in \bbR ^2\times \bbS$, and $\gd \in \bbR$. This a priori rather involved choice  
has the peculiarity that it retains the solvable character of the original Kuramoto model, if suitable
independence hypotheses are made among the components of the random vector $\go$. Again, to satisfy the boundedness assumptions, we restrict to the case in which the law of the first three component of $\go$ has bounded support.
\end{itemize}

\subsection{The nonlinear diffusion and the Fokker-Planck equation}
We introduce also the nonlinear diffusion  (degenerate, if $\gs=0$)
$\{\bar \theta^\go_t\}_{t\ge 0}$ that solves
\begin{equation}
\label{eq:theta_nonlin}
\qquad \dd\bar\theta^{ \omega}_{t}\,=\,  F\left(\bar \theta^{ \omega}_{ t}, \go\right)\dd t +  p \int \Gamma\left(\bar\theta_{ t}^{ \omega}, \omega, \tilde\theta, \tilde \omega\right) \nu_{ t}(\dd \tilde\theta, \dd \tilde \omega) \dd t + \sigma\dd B_{t}\, ,
\end{equation}
with initial condition which is a square integrable random variable. 
Moreover, in \eqref{eq:theta_nonlin}, $p\in(0, 1]$ is a parameter to be chosen appropriately (see Theorem~\ref{prop:prop_chaos_sparse} below) and $\go$ is a random variable also, with $(\bar\theta^{ \omega}_{0}, \go)$ that is distributed like  the $(\theta_{j,0,}\go_j)$ variables of the main model -- hence the law of this couple is $\nu_0$ --  and $(\bar\theta^{ \omega}_{0}, \go)$ is independent of $B_\cdot$. 
For $t>0$, $ \nu_{ t}$ is the law of $ (\bar\theta_{ t}^{ \omega}, \go)$. Existence and uniqueness for this problem are treated in \cite{cf:L} (see also \cite{cf:Gartner,cf:lancellotti,cf:S}).

The link between \eqref{eq:odegene} and 
\eqref{eq:theta_nonlin} is made by considering
$\{\bar \theta^{\go_i}_{i,t}\}_{t\ge 0}$
for $i=1,\ldots, n$ defined by setting $\bar \theta^{\go_i}_{i,t}$
equal to $\bar \theta^{\go_i}_{t}$ that solves \eqref{eq:theta_nonlin}  with the same initial condition
$\theta_{i,0}$ as for \eqref{eq:odegene}.
Note that $\{\bar \theta^{\go_i}_{i,\cdot}\}_{i=1,\ldots, n}$ is a sequence of independent processes 
with identical distribution. We will see that, under suitable conditions, 
$\{ \theta_{i,\cdot}\}_{i=1,\ldots, n}$ and 
$\{\bar \theta^{\go_i}_{i,\cdot}\}_{i=1,\ldots, n}$ stay close for {\sl arbitrarily large} positive times, for $n \to \infty$.

It is well known \cite{cf:dPdH,cf:L,cf:Gartner,cf:S} that $\nu_t$ solves (the weak form of) a Fokker-Planck (F-P) PDE.
It is definitely more transparent  to write this evolution when the  solution is classical, so let us assume that
$\nu_{ t}(\dd\theta, \dd \go)$ can be written as $q_t(\theta, \go) \dd \theta \nu_{\textrm{dis}}(\dd \go)$ where $\nu_{\textrm{dis}}$
is the common distribution of the $\go_j$ variables: of course, since there is no dynamics in the $\go$ variables, the marginal law
 $\int_\bbR \nu_{ t}(\dd \theta, \cdot)$ is independent of $t$ and it coincides with $\nu_{\textrm{dis}}$. When $ \sigma>0$, the parabolic character 
of the problem does ensure existence and regularity of $q_t(\theta, \go)$ solution to the F-P PDE \eqref{eq:FP} for positive times:
\begin{multline}
\label{eq:FP}
\partial_t q_t(\theta, \go)\, =\\ \frac{\gs^2}2 \partial_\theta^2 q_t(\theta, \go) - \partial_\theta \left[ q_t(\theta, \go) F(\theta, \go)\right]-
\partial_\theta \left[ q_t(\theta, \go) \int p\Gamma( \theta, \go, \theta', \go') q_t(\theta', \go') \dd \theta' \nu_{\textrm{dis}}(\dd \go') \right]\, .
\end{multline}
The same expression remains valid when $ \sigma=0$ if one chooses a sufficiently regular initial condition for \eqref{eq:FP}.
\subsection{The main result}
\label{sec:main}

In our main result $\xi$ is not random. 
We couple the two processes $ \theta_{i,\cdot}$ and $\bar \theta^{\go_i}_{i,\cdot}$ by using the same Brownian motion and we recall that we have chosen  $\bar \theta^{\go_i}_{i,0}= \theta_{i,0}$ for every $i$ and that 
$\{ (\theta_{i,0}, \go_i)\}_{i=1,2 , \ldots}$ is IID.  For conciseness we write  $\bar \theta_{i,t}$ for $\bar \theta^{\go_i}_{i,t}$.

\medskip
\begin{theorem}
\label{prop:prop_chaos_sparse}
Suppose that there exists $p\in (0, 1]$ such that 
\begin{equation}
\label{eq:deg-cond}
 b_n\,= \, b_n(\xi)\, :=\, \sup_{ i\in\{1, \ldots, n\}} \left\vert\frac{ \alpha_{ n}}{ n} \sum_{ k=1}^{ n} \left( \xi^{(n)}_{ i, k} - \frac{ p}{ \alpha_{ n}}\right)\right\vert\,=\,
 \sup_{ i\in\{1, \ldots, n\}} \left\vert
\ga_n \frac{d_i^{(n)}}n -p \right \vert\, 
  \stackrel{n \to \infty} \longrightarrow 0\, .
\end{equation}
Then there exist  $C=C(F, \Gamma)>0$ and $n_0\in \bbN$ (explicit, and depending only on $\{b_n\}_{n=2,\ldots}$ and $p$) such that for every $n >n_0$ and every $t\ge 0$
\begin{equation}
\label{eq:u}
\sup_{ i\in \{1, \ldots,n\}} \bbE \mathbf{ E} \left[ \sup_{ s \in [0, t]} \left\vert \theta_{ i, s} -\bar \theta_{ i, s} \right\vert^{ 2}\right]\, \leq \,  \left( \frac{\ga_n}n + b_n^2 \right) \left(\exp\left( Ct\right)-1 \right)\, .
\end{equation}
\end{theorem}

\medskip

This result is saying that, under the condition \eqref{eq:deg-cond} and under the assumptions we have made on the initial conditions,  each component of the $n$ dimensional diffusion \eqref{eq:odegene} is very close to the nonlinear diffusion process \eqref{eq:theta_nonlin} 
uniformly up to times that can even diverge (slowly)  with $n$. In order to understand the role of $\ga_n$ let us stress from now that 
we naturally distinguish two cases among the graphs we consider:
\smallskip 
\begin{enumerate}
\item The case of a graph $G^{ (n)}$ with {\sl positive edge density}: by this we mean that 
\begin{equation}
\sup_{ i=1, \ldots, n} \left\vert \frac{d_i^{(n)}} n - p \right\vert \stackrel{ n\to\infty}\longrightarrow 0\, ,
\end{equation} 
for some $p\in(0, 1]$. In this case, we just set $\ga_n=1$ for every $n$. The case of the complete graph fits into this framework for $p=1$.
\item The case of a graph $G^{ (n)}$ with {\sl vanishing edge density}, that is $d_i^{(n)}/n \stackrel{ n\to\infty}\longrightarrow 0$. In this case, we require that for a suitable choice of the sequence $\{\ga_n\}_{n=2, 3 , \ldots}$  we have that, for some $p>0$,
\begin{equation}
 \sup_{ i=1, \ldots, n}\left\vert \frac{\ga_n d_i^{(n)}}n -p \right\vert\stackrel{ n\to\infty}\longrightarrow 0\,.
 \end{equation}
It is therefore clear that in this case the value of $p$ is  completely conventional and we can and will choose it equal to $1$ in all vanishing edge density cases. Obviously, the only graphs $G^{ (n)}$ for which \eqref{eq:u} is meaningful are such that $ \frac{ \alpha_{ n}}{ n} \to 0$, as $n\to\infty$ (that is $d_{ n}^{ (i)}\to\infty$, uniformly in $i$).
\end{enumerate}

\smallskip

Therefore we can apply Theorem~\ref{prop:prop_chaos_sparse} to the case of complete graphs, so $b_n=0$ and $\ga_n=1$ for every $n$, and  the left-hand side in \eqref{eq:u} tends to zero when $t\leq c \log n$, for every choice of $c < 1/C$. 
As we will see in the next section, this is true, with a suitable choice of $c$ and  in an almost sure sense, for positive edge density Erd\H{o}s-R\'enyi graphs and even for {\sl most of} the vanishing edge density Erd\H{o}s-R\'enyi  graphs. 

\medskip 

 Theorem~\ref{prop:prop_chaos_sparse} directly implies some estimates on the empirical measure of the system.
 The empirical measure is
 \begin{equation}
 \label{eq:emp}
 \mu_{n,t}(\dd \theta, \dd \go)\, :=\, \frac 1n \sum_{i=1}^n \gd_{\theta_{ i,t}, \go_i}(\dd \theta, \dd \go)\,.
 \end{equation}
It is straightforward to see that  Theorem~\ref{prop:prop_chaos_sparse} implies 
the convergence in law of $\mu_{n,\cdot}$ to $ \nu_{\cdot}$ (that solves \eqref{eq:FP}), when $\mu_{n,\cdot}$ and $ \nu_{\cdot}$ are seen as (random) continuous functions
from $[0, T]$ (any $T>0$) to the set of probability measures. But we state more precisely a result that is more in the spirit of this note.
Consider the bounded Lipschitz distance between two probability measures defined on the same space metric space $(M, d)$
\begin{equation}
\label{eq:bL} 
d_{\mathrm{bL}}(\mu, \nu)\, :=\, \sup_{H \in \cL } \left \vert \int H \dd \mu - \int H \dd \nu \right\vert \, ,
\end{equation}
where $\cL$ is the set of $H:M \to [0,1]$ such that $\vert H(x)-H(y)\vert \leq d(x,y)$.

\medskip

\begin{corollary}
\label{th:emp}
Under the same assumptions of Theorem~\ref{prop:prop_chaos_sparse} and with the same $C$ as in that statement, if $\{t_n\}_{n=2, 3, \ldots}$
is such that 
$( (\ga_n/n) + b_n^2 ) \exp( Ct_n) = o(1)$  we have 
\begin{equation}
\lim_{n \to \infty} \bbE \bE \sup_{t \in [0, t_n]} d_{\mathrm{bL}}(\mu_{n, t}, \nu_t)\, =\, 0\,.
\end{equation}
\end{corollary}
\medskip

To resume, the content of Theorem~\ref{prop:prop_chaos_sparse} and Corollary~\ref{th:emp} can be viewed as twofold:
under degree conditions on $G^{(n)}$
\begin{enumerate}
\item dynamical models on $G^{(n)}$ are faithfully described for $n\to \infty$ by nonlinear diffusions or, equivalently, by F-P PDEs;
\item dynamical models on $G^{(n)}$ have a behavior that is close to the one of the same models on complete graphs, up to suitable rescaling of the interaction term. 
\end{enumerate}

\subsection{Examples I: mean field Fokker-Planck PDEs faithfully describe $n\to \infty$ systems}
\label{sec:examples}

\subsubsection{A family of graphs that includes Erd\H{o}s-R\'enyi graphs}  
Suppose here that, under $ \mathbb{ P}_{ \xi}$, for every vertex $i=1, \ldots, n$, the vector $(\xi_{i, 1}^{ (n)}, \ldots, \xi_{i, n}^{ (n)})$ is made of independent Bernoulli random variables with parameter $q_{ n}\in[0, 1]$. 

\medskip
\begin{proposition}
\label{prop:ER}
Under the above assumptions, 
$\lim_{n \to \infty} b_n(\xi)=0$
for $ \mathbb{ P}_{ \xi}$-almost every realizations of the graph $G^{ (n)}$
\begin{enumerate}
\item in the positive density case, that is when $q_n \stackrel{n\to \infty} \longrightarrow p \in (0,1]$.
\item in the zero density case, that is when $q_n \stackrel{n\to \infty} \longrightarrow 0$ and $ \alpha_{ n}:=1/q_n$, under the assumption that
\begin{equation}
\label{eq:cond_alpha_n}
\alpha_{ n}= o \left( \frac{ n}{ \log(n)}\right), \text{ as } n\to\infty.
\end{equation}
\end{enumerate}
\end{proposition}

\medskip

\begin{proof}[Proof of Proposition~\ref{prop:ER}]
Under the assumptions of Proposition~\ref{prop:ER}, the degree $d_i=d_i^{(n)}\sim Bin(n,q_n)$. Let us just recall the standard bound that follows from the Markov inequality: if $X_n$ is a binomial random variable of parameters $q$ and $n$ we have that for every $\gep\ge 0$ such that $q+\gep<1$
\begin{equation}
\label{eq:Bnp}
\bbP\left( X_n \, >\, (q+\gep) n \right) \, \leq \, 
\exp\left(-n \, D(q+ \gep\Vert q)\right)\, ,
\end{equation}
where
\begin{equation}
D(x \Vert y)\, := \, x\ln \left( \frac{ x}{ y}\right) + (1-x) \ln \left(\frac{ 1-x}{ 1-y}\right)\, ,
\end{equation}
for $x, y \in(0,1)$. Moreover \eqref{eq:Bnp} holds also for $\gep \leq 0$ if we reverse the inequality in
 $X_n  > (q+\gep) n$. By using the fact that the second derivative with respect to $\gep$ of
 $D((1+ \gep)q\Vert q)$, with $q \in (0,1)$ and $\gep $ such that $(1+ \gep)q\in (0,1)$, is
 $q/((1+\gep)(1-q(1+ \gep)))$, we see that $D((1+ \gep)q\Vert q)\ge c_q \gep^2$ 
with $c_q>0$ which depends continuously on $q$. We are dealing also with cases in which $q \searrow 0$: 
 $c_q$ can be chosen so that $c_q \ge q/4$  for $q \leq 1/2$ and $\gep \in [-1,1]$.  From this we readily obtain 
 that in the positive density case ($q_n \to p\in (0,1)$) 
 \begin{equation}
 \bbP_\xi\left( \vert {d_j^{(n)}} -n q_n\vert \, \ge \, \gep n \right) \, \leq \, 2 \exp( - n c_{q_n} \gep^2)\, ,
 \end{equation}
 and by choosing for example $\gep=\gep_n=n^{-c}$, $c\in (0,1/2)$, a union bound shows that $\lim_n b_n(\xi) =0$,
 $\bbP_\xi(\dd \xi)$-a.s..
In the zero density case instead ($q_n\to 0$, $\ga_n=1/q_n$ and $p=1$)
we have
\begin{equation}
 \bbP_\xi\left(\vert  \ga_n{d_j^{(n)}} -n \vert \, \ge \, \gep n \right)
 \, =\, 
  \bbP_\xi\left( \vert {d_j^{(n)}} -n q_n \vert \, \ge \, \gep n q_n \right) 
   \, \leq \, 2 \exp( - n c_{q_n} \gep^2)\, ,
 \end{equation}
and we see that if $\sum_n n \exp( - n q_n \gep_n^2/4)< \infty$
then  $\bbP_\xi(\dd \xi)$-a.s. there exists $n_0(\xi)<\infty$ such that  $b_n(\xi) \leq \gep _n$ 
for $n$ larger than $n_0(\xi)$. In particular  $1/q_n =o(n/\log n)$ ensures that we can find $\gep_n=o(1)$ such that $b_n(\xi) \leq \gep _n$  a.s. for $n$ sufficiently large. In fact if we set  $\gep^2_{ n}=4 c\frac{ \log n}{ n q_{ n}}=o(1)$ for a $c>2$, we have
 $\exp (- n q_{ n}  \varepsilon_{ n}^{ 2}/4)= \exp(- c\log(n))= { n^{- c}}$, so that $\sum_n n \exp( - n q_n \gep_n^2/4)< \infty$. 
\end{proof}
\medskip

Note that the assumptions of Proposition~\ref{prop:ER} do not require any independence between the vectors $(\xi_{i, 1}^{ (n)}, \ldots, \xi_{i, n}^{ (n)})$ for different $i$. If we assume this extra independence we are dealing with  Erd\H{o}s-R\'enyi multi-graphs: all the edges are independent, but
$\xi$ may have ones on the diagonal, so there are self loops. However it is straightforward to see that 
if $b_n(\xi)\to 0$ in the case we just dealt with, the same is true if we erase the self loops. Moreover
the analysis we have presented  covers only the case of asymmetric Erd\H{o}s-R\'enyi graphs, but the analysis of the symmetric case is 
almost identical. 
  
\subsubsection{Random regular graphs} 
One can build a graph in which every vertex has degree $d=d(n)$ provided $3\leq d<n$ and $dn$ even 
 \cite[Sec.~2.4]{cf:bollobas} and our results apply to these sequences provided that $\lim_{n\to \infty }d(n)=\infty$: in the vanishing density case 
 we can choose $\ga_n=n/d(n)$, so $b_n=0$ for every $n$ (cf. \eqref{eq:deg-cond}), and $d(n)$ must tend to infinity, i.e. $\lim_n \ga_n/n=0$, for \eqref{eq:u} to bear relevant information. For a performing algorithm choosing one of this regular graphs at random see e.g. \cite{10.1371/journal.pone.0010012,MCKAY199052,Britton2006} and references therein.   

\subsection{Examples II: Fokker-Planck PDEs may not faithfully reproduce the $n$ large (but finite) system on graphs}
\label{sec:pathologies} 
We now argue that, in spite of being appealing on a first look, our main results have strong limitations when we consider them from a {\sl practical} viewpoint, that is for $n$ large but on a {\sl reasonable} scale (e.g., the ones of simulations or of the true systems to which we would like to apply them, like for families of interacting cells). The problem is already present at a conceptual level:
Theorem~\ref{prop:prop_chaos_sparse} does not require the graph to be connected. We develop in some detail a specific case with two connected components, but it will be clear to the reader that several generalizations are possible. 

Consider the easiest set up, that is non disordered Kuramoto model (i.e., the  dynamical version of the mean field plane rotator model): $F\equiv 0$ and 
$\Gamma(\theta, \go, \theta', \go')= -K \sin (\theta-\theta')$ for every $\theta, \theta', \go$ and $\go'$. Let us assume for simplicity that $ \sigma=1$. We choose $n=2N$
and $\xi_{i,j}^{(n)}=1$ if and only if both $i$ and $j$ are in $\{1,\ldots, N\}$ or both are in $\{N+1,\ldots, 2N\}$. So $G^{(n)}$ 
has two connected components and both components are completely connected. $G^{(n)}$ satisfies \eqref{eq:deg-cond} with $p=1/2$ (more: $b_{2N}(\xi)=0$ for every $N$). This fact is of course rather trivial and it follows by analyzing two independent $N$ dimensional mean field systems. The F-P PDE   \eqref{eq:FP} in this case just reads
\begin{equation} 
\label{eq:FPK}
\partial_t q_t(\theta)\, \, = 
\frac{1}2 \partial_\theta^2 q_t(\theta) -\partial_\theta \left[ q_t(\theta)J* q_t (\theta) \right]
\, ,
\end{equation}
where $J(\cdot)=-K\sin(\cdot)/2$. This equation is studied in detail in the literature - see in particular \cite{cf:GPP} and references therein
- but the important point for us is that all stationary solutions of \eqref{eq:FPK} that are probability densities can be written
as $\theta \mapsto \frac{ 1}{ Z_{ K}} \exp(Kr_K \cos(\theta-\psi)$, where $\psi$ is an arbitrary phase, $Z_{ K}$ an appropriate normalization constant and $r_K\in[0, 1)$ is any solution of a suitable 
one dimensional map (e.g. \cite{cf:SFN,cf:BeGiPa}). This fixed point problem admits $0$ as a solution, this corresponds to the uniform probability density $1/(2\pi)$, but for $K>K_c=2$ it admits also a positive solution, which we just call $r_K$ and corresponds to a nontrivial probability density, corresponding to (partial) synchronization in the system.
For $K>2$  one can show \cite{cf:BeGiPa} suitable stability properties for the family of non trivial stationary solutions and it is straightforward to see that the uniform probability density is unstable. In fact if we rewrite \eqref{eq:FPK}
in terms of the Fourier coefficients 
\begin{equation}
\label{eq:Fourier}
c_k(t):=\int q_t(\theta) \cos(k\theta) \dd \theta \ \  \text{  and } \ \
s_k(t):=\int q_t(\theta) \sin(k\theta) \dd \theta,\, k\geq1\, ,
\end{equation}
one verifies that at a linear level all modes decouple and they solve the equation $\dot{x}= \gl_k x$, with $\gl_k=-k^2/2$  for every mode  $k\ge 2$ and $\gl_1=(K-2)/4$  (hence these are the two unstable modes, the zero mode is conserved). 
A more refined analysis, 
using the $N$ dimensional diffusion 
with initial condition given by IID random variables uniformly distributed on the circle,
shows that a good approximation for the large $N$ evolution
of the empirical mean $c_{1,N}(t):=\frac 1N \sum_{j=1}^N \cos(\theta_j(t))$ is the linear SDE
\begin{equation}
\label{eq:x}
\dd x(t)\, =\, \frac{(K-2)}4 x(t) \dd t + \frac 1{\sqrt{2N}} \dd B_t\, ,
\end{equation}
with initial condition $\cN(0, 1/(2N))$, independent of the (standard) Brownian motion $B_\cdot$. 
We present \eqref{eq:x} only in an informal way and a full treatment requires the fluctuation analysis of the system on times scales of order
$\log n$:
see e.g. \cite{cf:dPdH,cf:L} for fluctuation analysis on time scale of order one in the Kuramoto set-up and \cite{cf:CdP} for  a longer time analysis of the Kuramoto case, but not related to the question we are addressing here (for a full analysis of a case that is close in spirit to the one we are discussing --    in the context of reaction-diffusion particle systems -- see  \cite{cf:dMPV}). 
The same is true for the sine mode -- call $s_{1,N}(t)$ the corresponding empirical mean -- with  new (independent) Brownian motion and initial condition. These equations can be solved and the solution is a centered Gaussian process. Of particular interest is the {\sl synchronization} degree $r_N(t)=\sqrt{(c_{1,N}(t))^2+(s_{1,N}(t))^2}$ for which one computes (using the approximate evolution \eqref{eq:x})
\begin{equation}
\label{eq:rNt}
 \bbE\left[ (r_N(t))^2\right]\, \approx\, \frac 1N\exp (2\gl_1 t) \left(1+ \frac 1{2\gl_1}(1-\exp(-2\gl_1 t))\right)\,,
\end{equation}
for $N$ large and as far as the linear approximation is reliable. 
Again a full theory of the phenomenon we are describing is beyond the scope of this note, but, in analogy with \cite{cf:dMPV}, the fluctuations lead to the escape from the flat state at a time $a \log N$, $a:= (2\gl_1)^{-1}$ in the sense that for $t\leq  c \log N$, any $c \in (0, a)$, the empirical measure of the system converges for
$N \to \infty$ to the uniform probability, but for $t=c\log N$, any $c>a$, it converges to one of the {\sl synchronized} solutions
(see Fig.~\ref{fig:escape} and its caption). 

\begin{figure}[h]
\begin{center}
\leavevmode
\epsfxsize =11 cm
\psfragscanon
\psfrag{0}[c][0]{\tiny $0$}
\psfrag{5}[c][0]{\tiny $5$}
\psfrag{10}[c][0]{\tiny $10$}
\psfrag{15}[c][0]{\tiny $15$}
\psfrag{20}[c][0]{\tiny $20$}
\psfrag{pi}[c][l]{\tiny $\pi$}
\psfrag{mpi}[c][l]{\tiny $-\pi$}
\psfrag{t}[c][l]{\tiny $t$}
\psfrag{rt}[c][l]{\tiny $r_N(t)$}
\psfrag{06}[c][l]{\tiny $0.6$}
\psfrag{0.2}[c][c]{\tiny $0.2$}
\psfrag{0.4}[c][c]{\tiny $0.4$}
\psfrag{0.6}[c][c]{\tiny $0.6$}
\psfrag{0.8}[c][c]{\tiny $0.8$}
\epsfbox{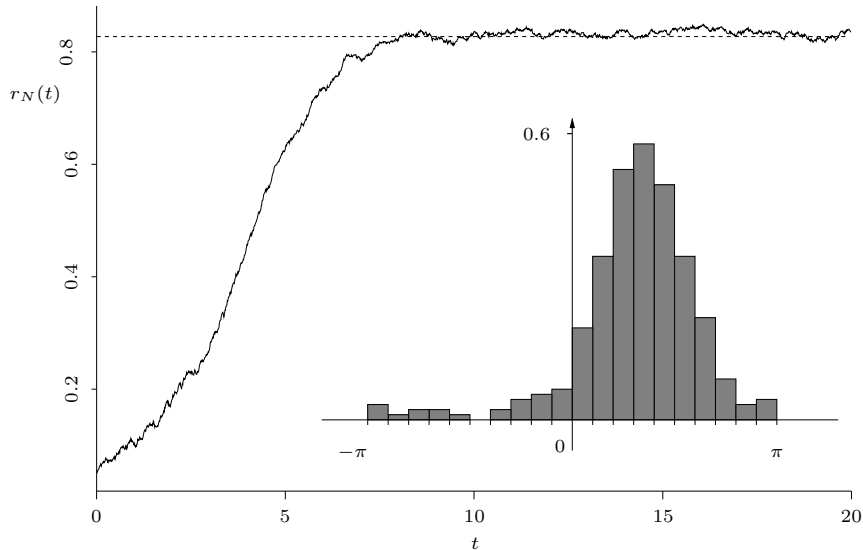}
\end{center}
\caption{\label{fig:escape} 
A numerical realization of the complete graph $N$ dimensional diffusion with $N=1000$, $K=4$ (so $\gl_1=1/2$) and  interaction potential $J(\cdot)=-(K/2)\sin(\cdot)$ to fit in the framework \eqref{eq:FPK}-\eqref{eq:rNt},  up to time $20$. This simulation deals therefore with only half of the system: the histogram  is the empirical measure limited to one complete graph and it is taken at time $t=20$, but the empirical measure is just about the same starting from $t\approx 7$, as suggested by the plot of the {\sl synchronization degree} $r_N(t)$. It is curious to note   that the
 right-hand side of \eqref{eq:rNt} becomes equal to one at $t=\log N /(2\gl_1) \approx 6.9$. Moreover, if it is evident that for 
 $t\ge7$ the system is very far from the initial {\sl stationary} profile $1/(2\pi)$, the deviation is already clear much earlier. Another crucial point that one easily guesses and verifies with the simulations is that the center of synchronization, see for example \cite{cf:BeGiPa} for  a precise definition, is uniformly distributed on the unit circle: in the realization we present the center is at about one radiant. When the whole system is composed by two independent and identically distributed $N$ dimensional diffusions, the flat initial condition will generate on the time scale $\log N$ an empirical measure that is the superposition of two empirical measures like the one in the figure, but with different centers (hence a bimodal distribution!). On the much longer time scale $t\propto N$ one can appreciate a random motion of the two synchronization centers \cite{cf:BGP,cf:Dahms}. 
}
\end{figure}

This calls for several considerations:

\medskip
\begin{enumerate}
\item Theorem~\ref{prop:prop_chaos_sparse} and Corollary~\ref{th:emp} cannot be substantially improved and, in general, there will be a breakdown on the time-scale $\log n$ of the information that it carries. This is particularly troublesome for the applications because
times of the order $\log n$ are essentially finite times from a numerical viewpoint, but also for many, possibly most of, real life systems.
\item Our result, as pointed out in the two item list at the end of Section~\ref{sec:main}, is twofold: let us reconsider these two items. For the first item,
we have seen that  we can have 
the breakdown of the reduction to F-P evolution, but  this  is not  surprising because it already happens on the complete graph. However  the second item is  more troublesome. In fact, the example we presented shows that on logarithmic time scale the evolution on general   graphs may have little to do 
with the complete graph counterpart, see Fig.~\ref{fig:escape} and its caption: the flat profile escapes to a bimodal distribution in the two component case we have studied, while on the complete graph the distribution becomes unimodal.
\item In spite of the breakdown of the proximity between the $n$-dimensional diffusion and the F-P limit, in various instances 
the analysis has been pushed well beyond finite or logarithmic times even in presence of phase transitions, that is  many stationary states for the limit F-P PDE: it is difficult to account for such a large literature, see e.g. \cite{cf:OV} for the case of multiple isolated and stable F-P stationary solutions in which 
stochastic deviations from F-P behavior happen on exponentially long times, see  e.g.  \cite{cf:BGP,cf:LP} (to limit ourselves to the Kuramoto context that we consider)  for cases in which multiple stable stationary solutions are not isolated and stochastic effects happen on much shorter time scales (power law in $n$), and of course in presence of unstable states the deviations are observed on the $\log n$ scale.
The example we have developed shows that going to longer time scales   cannot be done in our context without stronger hypotheses on the graph than just the degree condition
\eqref{eq:deg-cond}. So the issue is: how should one complement the uniform degree condition in order to push the proximity of the graph case
and the corresponding renormalized complete graph case to longer times? It is worth underlying that one of the obstacles to apply the iterative method used for example in  \cite{cf:BGP,cf:LP} is the fact that on a graph that is not complete the empirical measure
\eqref{eq:emp} does not carry all the information about the system.
\item Of course the previous question could be asked also in cases in which one expects results that are uniform in time
(for example in the sense of \cite[Sec.~3.2]{cf:CGM}, see \cite{cf:CGM} also for more references on this issue). In these cases it is not unreasonable to expect that our uniform degree condition could allow obtaining results that are uniform in time. 
\end{enumerate}

\medskip

\begin{remark}
In the example we have developed it is completely clear that the two-fold graph structure  becomes apparent only
on a long time scale because
 the initial condition we have chosen  is IID  and does not depend on the graph. 
This observation is elementary, but it reveals the fundamental difficulty in going beyond the results 
we have presented: the evolution does depend on the graph and so the state of the system 
at positive times  depends on the graph. Dealing with this dependence is the challenge.
\end{remark}

\medskip

Even if this note raises questions more than it gives answers we want to conclude this section by pointing out one more open issue: the case   of sparse graphs, that is graphs with bounded degree.
For the graphs we consider in this work  each degree $d_{ i}^{ (n)}$ diverges as $n\to\infty$ and our results  do
include {\sl almost all} Erd\H{o}s-R\'enyi graphs in this class,
leaving out the open issue of the regime where $d_{ i}^{ (n)}=o(\log(n))$ (recall \eqref{eq:cond_alpha_n}). The issue becomes particularly intriguing in the sparse case, 
say Erd\H{o}s-R\'enyi, that is $q_n \sim c/n$ and $ \mathbb{ E}(d_{ i}^{ (n)})\sim c$, when there is an infinite component (so $c>1$). The case of sparse graphs has been treated in details in equilibrium statistical mechanics  \cite{cf:MM,cf:DM}, but, to our knowledge, mostly for discrete spin models.

\section{Proofs}
\begin{proof}[Proof of Theorem~\ref{prop:prop_chaos_sparse}]
For every $s>0$
\begin{multline}
\left\vert \theta_{ i, s} - \bar \theta_{ i, s} \right\vert^{ 2} =  2 \int_{ 0}^{s} \left( \theta_{ i, u} - \bar \theta_{ i, u}\right) \left(F( \theta_{ i, u}, \go_i) - F( \bar \theta_{ i, u}, \go_i)\right) \dd u\\ + 2\int_{ 0}^{s}  \left( \theta_{ i, u} - \bar \theta_{ i, u}\right) \left(\frac{ \alpha_{ n}}{n} \sum_{ k=1}^{ n} \xi_{ i, k}^{ (n)} \Gamma( \theta_{ i, u}, \omega_{ i}, \theta_{ k, u}, \omega_{ k}) - p\int \Gamma(\bar \theta_{ i, u}, \omega_{ i}, \tilde\theta, \tilde \omega) \nu_{ u}(\dd \tilde\theta, \dd \tilde \omega)\right) \dd u\, .
\end{multline}
Using the Lipschitz-continuity of $F$, $L_F$ is the Lipschitz constant, we have
\begin{equation}
\begin{split}
\left\vert \theta_{ i, s} - \bar \theta_{ i, s} \right\vert^{ 2} &\leq  \left(2L_F +1\right) \int_{ 0}^{s}\left\vert  \theta_{ i, u} - \bar \theta_{ i, u}\right\vert^{ 2} \dd u\\
& + \int_{0}^{s} \left\vert\frac{ \alpha_{ n}}{ n} \sum_{ k=1}^{ n} \xi_{ i, k}^{ (n)} \Gamma( \theta_{ i, u}, \omega_{ i}, \theta_{ k, u}, \omega_{ k}) - p\int \Gamma(\bar \theta_{ i, u}, \omega_{ i}, \tilde\theta, \tilde \omega) \nu_{ u}(\dd \tilde\theta, \dd \tilde \omega
)\right\vert^{ 2}\dd u\, .
\end{split}
\end{equation}
Taking the supremum in $s\in[0, t]$ and the expectation $\mathbf{ E}[\cdot]$ w.r.t. the Brownian motions 
\begin{multline}
\mathbf{ E} \left[ \sup_{ s \in [0, t]}\left\vert \theta_{ i, s} - \bar \theta_{ i, s} \right\vert^{ 2}\right] \leq  \left(2L_F +1\right) \int_{ 0}^{t} \mathbf{ E} \left[ \left\vert  \theta_{ i, u} - \bar \theta_{ i, u}\right\vert^{ 2}\right] \dd u\\ + \int_{0}^{t} \mathbf{ E} \left[ \left\vert\frac{ \alpha_{ n}}{ n} \sum_{ k=1}^{ n} \xi_{ i, k}^{ (n)} \Gamma( \theta_{ i, u}, \omega_{ i}, \theta_{ k, u}, \omega_{ k}) - p\int \Gamma(\bar \theta_{ i, u}, \omega_{ i}, \tilde\theta, \tilde \omega) \nu_{ u}(\dd \tilde\theta, \dd \tilde \omega)\right\vert^{ 2}\right]\dd u\, ,
\end{multline}
so that
\begin{multline}
 \label{eq:diff_thetai_B}
\mathbf{ E} \left[ \sup_{ s \in [0, t]}\left\vert \theta_{ i, s} - \bar \theta_{ i, s} \right\vert^{ 2}\right] \leq  (2L_F +1) \int_{ 0}^{t}\mathbf{ E} \left[ \sup_{ v \in [0, u]}\left\vert  \theta_{ i, v} - \bar \theta_{ i, v}\right\vert^{ 2}\right] \dd u\\ + \int_{0}^{t} \mathbf{ E} \left[ \left\vert\frac{ \alpha_{ n}}{ n} \sum_{ k=1}^{ n} \xi_{ i, k}^{ (n)} \Gamma( \theta_{ i, u}, \omega_{ i}, \theta_{ k, u}, \omega_{ k}) - p \int \Gamma(\bar \theta_{ i, u}, \omega_{ i}, \tilde\theta, \tilde \omega) \nu_{ u}(\dd \tilde\theta, \dd \tilde \omega)\right\vert^{ 2}\right]\dd u\, .
\end{multline}
The last term in \eqref{eq:diff_thetai_B} can be bounded  above by $4 \int_{ 0}^{t}\left(\sum_{ k=1}^{ 4}B_{ n, i, s}^{ (k)}\right)\dd s$, where
\begin{equation}
\begin{split}
B_{ n, i, s}^{ (1)}&= \mathbf{ E} \left\vert \frac{ \alpha_{ n}}{ n} \sum_{ k=1}^{ n} \xi_{ i, k}^{ (n)} \left(\Gamma( \theta_{ i, s}, \omega_{ i}, \theta_{ k, s}, \omega_{ k})- \Gamma( \bar\theta_{ i, s}, \omega_{ i}, \theta_{ k, s}, \omega_{ k})\right) \right\vert^{ 2},\\
B_{ n, i, s}^{ (2)}&= \mathbf{ E} \left\vert \frac{ \alpha_{ n}}{n} \sum_{ k=1}^{ n} \xi_{ i, k}^{ (n)} \left(\Gamma( \bar\theta_{ i, s}, \omega_{ i}, \theta_{ k, s}, \omega_{ k}) - \Gamma( \bar\theta_{ i, s}, \omega_{ i}, \bar\theta_{ k, s}, \omega_{ k})\right) \right\vert^{ 2},\\
B_{ n, i, s}^{ (3)}&= \mathbf{ E} \left\vert \frac{ \alpha_{ n}}{ n} \sum_{ k=1}^{ n} \xi_{ i, k}^{ (n)} \left(\Gamma( \bar\theta_{ i, s}, \omega_{ i}, \bar\theta_{ k, s}, \omega_{ k}) - \int \Gamma( \bar\theta_{ i, s}, \omega_{ i}, \tilde\theta, \omega) \nu_{ s}(\dd \tilde\theta ,\dd \omega) \right) \right\vert^{ 2},\\
B_{ n, i, s}^{ (4)}&= \left(\frac{ \alpha_{ n}}{n} \sum_{ k=1}^{ n} \left(\xi_{ i, k}^{ (n)} - \frac{ p}{ \alpha_{ n}}\right)\right)^{ 2} \mathbf{ E} \left\vert\int \Gamma( \bar\theta_{ i, s}, \omega_{ i}, \tilde\theta, \omega) \nu_{ s}(\dd \tilde\theta, \dd \omega) \right\vert^{ 2}\, .
\end{split}
\end{equation}
For what concerns  $B^{ (1)}$ we have (with $L_\Gamma$ for the Lipschitz constant of $\Gamma$)
\begin{equation}
\begin{split}
\bbE \left(B_{ n, i, s}^{ (1)}\right) &\leq  L_\Gamma^{ 2} \left(\frac{ \alpha_{ n}}{ n} \sum_{ k=1}^{ n} \xi_{ i, k}^{ (n)}\right)^{ 2} \sup_{ r \in \{1, \ldots, n\} }\bbE\mathbf{ E} \left[\sup_{ u \in [0, s]}\left\vert \theta_{ r, u} - \bar \theta_{ r, u} \right\vert^{ 2}\right]\,.
\end{split}
\end{equation}
Concerning $B^{ (2)}$
\begin{equation}
\begin{split}
\bbE \left(B_{ n, i, s}^{ (2)}\right) &\leq L_\Gamma^{ 2} \bbE\mathbf{ E} \left[ \left(\frac{ \alpha_{ n}}{n} \sum_{ k=1}^{ n} \xi_{ i, k}^{ (n)} \left\vert \theta_{ k, s} - \bar\theta_{ k, s} \right\vert\right)^{ 2}\right]\\
&= L_\Gamma^{ 2}\left(\frac{ \alpha_{ n}}{n}\right)^{ 2} \sum_{ k,l=1}^{ n} \xi_{ i, k}^{ (n)}\xi_{ i, l}^{ (n)} \bbE\mathbf{ E} \left[ \left\vert \theta_{ k, s} - \bar\theta_{ k, s} \right\vert \left\vert \theta_{ l, s} - \bar\theta_{ l, s} \right\vert\right]\\
&\leq L_\Gamma^{ 2}\left(\frac{ \alpha_{ n}}{n}\right)^{ 2} \sum_{ k,l=1}^{ n} \xi_{ i, k}^{ (n)}\xi_{ i, l}^{ (n)} \frac{ 1}{ 2}  \bbE\mathbf{ E} \left[\left\vert \theta_{ k, s} - \bar\theta_{ k, s} \right\vert^{ 2} + \left\vert \theta_{ l, s} - \bar\theta_{ l, s} \right\vert^{ 2}\right]\\
& \leq L_\Gamma^{ 2} \left(\frac{ \alpha_{ n}}{ n}  \sum_{ k=1}^{ n} \xi_{ i, k}^{ (n)} \right)^{ 2} \sup_{ r \in \{1, \ldots, n\}}\bbE
\mathbf{ E} \left[\sup_{ u \in [0, s]}\left\vert \theta_{ r, u} - \bar\theta_{ r, u} \right\vert^{ 2} \right]\, .
\end{split}
\end{equation}
Concerning $B^{ (3)}$, first denote by 
\begin{equation}
\bar \Gamma_{ s}^{\theta, \omega}(\theta^{ \prime}, \omega^{ \prime})= \Gamma(\theta, \omega, \theta^{ \prime}, \omega^{ \prime}) - \int \Gamma(\theta, \omega, \tilde\theta, \tilde\omega) \nu_{ s}(\dd \tilde\theta,\dd \tilde\omega)\, ,
\end{equation} 
so that
\begin{equation}
\begin{split}
B_{ n, i, s}^{ (3)}&= \frac{ \alpha_{ n}^{ 2}}{ n^{ 2}} \sum_{ k,l=1}^{ n }\xi_{ i, k}^{ (n)} \xi_{ i, l}^{ (n)} \mathbf{ E} \left[\bar\Gamma_{ s}^{ \bar \theta_{ i, s}, \omega_{ i}}(\bar \theta_{ k, s}, \omega_{ k})\bar\Gamma_{ s}^{\bar \theta_{ i, s}, \omega_{ i}}(\bar \theta_{ l, s}, \omega_{ l})\right]\, .
\end{split}
\end{equation}
This gives
\begin{equation}
\label{eq:Gabar}
\begin{split}
\bbE \left(B_{ n, i, s}^{ (3)}\right)&= \frac{ \alpha_{ n}^{ 2}}{ n^{ 2}} \sum_{ k,l=1}^{ n }\xi_{ i, k}^{ (n)} \xi_{ i, l}^{ (n)} \bbE\mathbf{ E} \left[\bar\Gamma_{ s}^{ \bar \theta_{ i, s}, \omega_{ i}}(\bar \theta_{ k, s}, \omega_{ k})\bar\Gamma_{ s}^{\bar \theta_{ i, s}, \omega_{ i}}(\bar \theta_{ l, s}, \omega_{ l})\right]\\
&= \frac{ \alpha_{ n}^{ 2}}{ n^{ 2}} \sum_{ k=1}^{ n } \left(\xi_{ i, k}^{ (n)}\right)^{ 2} \bbE\mathbf{ E} \left[\bar\Gamma_{ s}^{ \bar \theta_{ i, s}, \omega_{ i}}(\bar \theta_{ k, s}, \omega_{ k})^{ 2}\right]\\
&\ \ + \frac{ \alpha_{ n}^{ 2}}{ n^{ 2}} \sumtwo{k, l\in \{1, \ldots, n\}:}{ k\neq l}\xi_{ i, k}^{ (n)} \xi_{ i, l}^{ (n)} \bbE\mathbf{ E} \left[\bar\Gamma_{ s}^{ \bar \theta_{ i, s}, \omega_{ i}}(\bar \theta_{ k, s}, \omega_{ k})\bar\Gamma_{ s}^{\bar \theta_{ i, s}, \omega_{ i}}(\bar \theta_{ l, s}, \omega_{ l})\right].
\end{split}
\end{equation}
By construction of $\bar \Gamma$, we see that the 
term in the last line of \eqref{eq:Gabar}
 is zero: just take the expectation with respect to $(\bar \theta_{ k, s}, \omega_{ k})$ if $k\neq i$,
 otherwise do the same but with respect to $(\bar \theta_{ l, s}, \omega_{ l})$.  Therefore
\begin{equation}
\begin{split}
\sup_{ i \in \{1, \ldots n\}} \bbE \left(B_{ n, i, s}^{ (3)}\right)&\leq 4 \left\Vert \Gamma \right\Vert_{ \infty}^{ 2} \left(\frac{ \alpha_{ n}^{ 2}}{ n^{ 2}} \sum_{ k=1}^{ n } \xi_{ i, k}^{ (n)}\right).
\end{split}
\end{equation}
Concerning $B^{ (4)}$, one has obviously
\begin{equation}
\bbE\left(B_{ n, i, s}^{ (4)}\right) \leq \left\Vert \Gamma \right\Vert_{ \infty}^{ 2}\left(\frac{ \alpha_{ n}}{n} \sum_{ k=1}^{ n} \left(\xi_{ i, k}^{ (n)} - \frac{ p}{ \alpha_{ n}}\right)\right)^{ 2}\, .
\end{equation}
Taking the expectation $ \bbE$ and the supremum over $i\in \{1, \ldots, n\}$ in \eqref{eq:diff_thetai_B}, we obtain
\begin{equation}
\label{eq:gronwall_ut}
u_{ t} \leq c\int_{ 0}^{t}  \left(1+ a_{ n}^{ 2}\right) u_{ s}\dd s +c t \left(\frac{ \alpha_{ n}}{ n} a_{ n} + b_{ n}^{ 2}\right), 
\end{equation}
 where we used the notation $u_t$ for the left-hand side in 
\eqref{eq:u}, $c= \max(2L_F+1, 2L_\Gamma^2, 4 \Vert \Gamma\Vert_\infty^2)$, $b_n$ is defined in \eqref{eq:deg-cond} and
\begin{equation}
a_{ n}\,=\, a_{ n}(\xi)\,:=\, \sup_{ i \in \{1, \ldots n\}} \left(\frac{ \alpha_{ n}}{ n} \sum_{ k=1}^{ n} \xi_{ i, k}^{ (n)}\right)
\, .
\end{equation}
Note that \eqref{eq:deg-cond} implies that $\lim_{n \to \infty} a_n = p$.
Inequality \eqref{eq:gronwall_ut} gives
\begin{equation}
u_{ t} \leq  \left(\frac{\frac{ \alpha_{ n}}{ n}a_{ n} + b_{ n}^{ 2}}{ 1+ a_{ n}^{ 2}}\right) \left(\exp \left(c(1+ a_{ n}^{ 2})t\right) -1\right)
\, ,
\end{equation}
and since $a_n \leq {2}p$ for $n$ sufficiently large (how large depends only on how fast $b_n$ tends to zero), we readily see  
that
\eqref{eq:u} holds with $C=3c$ and
the proof is complete. 
\end{proof}

\medskip

\begin{proof}[Proof of Corollary~\ref{th:emp}]
We start by observing that (see \cite[proof of Theorem~11.3.3]{cf:dudley}) for every $\gep>0$ there exists a finite set $\cL_\gep \subset \cL$ such that for every $\mu$ and $\nu$
\begin{equation}
\label{eq:bLbound}
d_{\mathrm{bL}}(\mu, \nu)\, \leq \, \gep +\max_{H \in \cL_\gep } \left \vert \int H \dd \mu - \int H \dd \nu \right\vert \, .
\end{equation}
So it suffices to show that 
$\bbE \bE \sup_{t \in [0, t_n]}\vert \int H \dd \mu_{n,t} - \int H \dd \nu_{t}\vert$ vanishes as $n\to \infty$ for an arbitrary choice of $H\in \cL$.
For this we remark that 
\begin{multline}
\left\vert \int H \dd \mu_{n,t} - \int H \dd \nu_{t}\right\vert \,=\\
\left \vert  \frac 1n \sum_{i=1}^n \left[H( \theta_{ i, t}, \go_i)-H(\bar \theta_{ i, t}, \go_i)\right]+
\frac 1n \sum_{i=1}^n \left[H(\bar \theta_{ i, t}, \go_i)- \bbE \bE H(\bar \theta_{ i, t}, \go_i)\right] \right \vert
\\
\,
\le \frac 1n  \sum_{i=1}^n \left \vert \theta_{i,t}- \bar \theta_{ i, t} \right \vert +
\left \vert \frac 1n \sum_{i=1}^n X_{i,t} \right \vert\, =:\, A^{(1)}_{n,t}+ A^{(2)}_{n,t}\, ,
\end{multline}
and we have introduced $X_{i,t}:= H(\bar \theta_{ i, t}, \go_i)- \bbE \bE H(\bar \theta_{ i, t}, \go_i)$. 
Since
\begin{equation}
\bbE \bE \sup_{t \le t_n}  A^{(1)}_{n,t} \, \le\, \sup_{i\in\{1, \ldots, n\}} \bbE \bE  \left[ \sup_{ t \in [0, t_n]} \left\vert \theta_{ i, t} -\bar \theta_{ i, t} \right\vert^{ 2}\right]^{1/2}\, ,
\end{equation}
 this term vanishes for $n\to \infty$ by  \eqref{eq:u} of Theorem~\ref{prop:prop_chaos_sparse}.

We are left with controlling the term containing $A^{(2)}_{n,t}$. For this note that $\{X_{i,t}\}_{i=1, \ldots, n}$ are IID random variables
that are bounded uniformly in $t$ (they are in fact bounded by one, since $H(\cdot)\in [0,1]$).
Therefore $\sup_{t\le t_n} \bbE \bE A^{(2)}_{n,t} \le 1/ \sqrt{n}$ and we are left with switching the supremum with the expectations.
For this we introduce a time discretization of step $n^{-1/4}$, i.e. $s_0=0$ and $s_{j+1}-s_j= n^{-1/4}$, and
$\ell_n:= \sup\{ j=0,1, \ldots:\, s_j \le t_n\}$.
Observe that since $\ga_n \ge 1$ eventually $t_n \le C^{-1} \log n$, so $\ell_n \le c n^{1/4} \log n$ for any $c>1/C$. 
We have
\begin{multline}
\bbE \bE \sup_{t \le t_n}  A^{(2)}_{n,t} \, \le\, \bbE \bE \sup_{j=0, \ldots, \ell_n }  A^{(2)}_{n,s_j} +
 \bbE \bE \sup_{j=0, \ldots, \ell_n } \sup_{s\in [0, n^{-1/4}]} \left \vert A^{(2)}_{n,s_j+s} -A^{(2)}_{n,s_j} \right \vert 
 \\
 \le \, \frac{\ell_n +1}{n^{1/2}} +
 \bbE \bE \sup_{j=0, \ldots, \ell_n } \sup_{s\in [0, n^{-1/4}]} \left \vert A^{(2)}_{n,s_j+s} -A^{(2)}_{n,s_j} \right \vert 
 \, , 
\end{multline}
and we are left with showing that the rightmost addend vanishes too. For this we recall that 
the random variables $\vert A^{(2)}_{n,s_j+s} -A^{(2)}_{n,s_j} \vert$ are uniformly bounded (they are bounded by two, because 
$\vert X_{i,t}\vert \le 1$). The $L^1$ estimate can therefore be replaced by an estimate in probability: it suffices to show that for
every $\gd>0$ we have 
\begin{equation}
\lim_{n \to \infty} \bbE \bP \left(\sup_{j=0, \ldots, \ell_n } \sup_{s\in [0, n^{-1/4}]} \left \vert A^{(2)}_{n,s_j+s} -A^{(2)}_{n,s_j} \right \vert \, > \gd 
\right)\, =\, 0\,,
\end{equation}
and, in turn, this is implied by 
\begin{equation}
\label{eq:probabest}
 \sup_{j=0, \ldots, \ell_n } \bbE \bP \left( \sup_{s\in [0, n^{-1/4}]} \left \vert A^{(2)}_{n,s_j+s} -A^{(2)}_{n,s_j} \right \vert \, > \gd
\right)\, =\, o\left(1/\ell_n\right)\,.
\end{equation}
For this we observe that $\sup_{s \in [0, n^{-1/4}]}\vert  X_{i,s_j+s}- X_{i,s_j}\vert$ 
is 
\begin{equation}
\label{eq:tobeE}
\sup_{s \in [0, n^{-1/4}]}\left \vert H(\bar \theta_{ i, s_j+s}, \go_i)-H(\bar \theta_{ i, s_j}, \go_i)- \bbE \bE 
\left[H(\bar \theta_{ i, s_j+s}, \go_i)-H(\bar \theta_{ i, s_j}, \go_i)\right] \right\vert
\end{equation}
and that $\vert H(\bar \theta_{ i, s_j+s}, \go_i)-H(\bar \theta_{ i, s_j}, \go_i)\vert $ is bounded by 
$\vert \bar \theta_{ i, s_j+s}-\bar \theta_{ i, s_j}\vert$, which, by using \eqref{eq:theta_nonlin}, is bounded by
$C_{F, \Gamma}n^{-1/4}+ \gs \vert B_{ i, s_j+s}- B_{ i, s_j}\vert$, with $C_{F, \Gamma}$ equal to $\max( \Vert F \Vert_\infty, \Vert \Gamma\Vert_\infty)$.
Therefore the expectation of the expression in 
\eqref{eq:tobeE} is bounded by 
\begin{equation}
2C_{F, \Gamma}n^{-1/4}+ 2
 \bE \left( \sup_{s\in [0, n^{-1/4}]}\left\vert B_s\right\vert \right)\, \le \, C_{F, \Gamma, \gs}n^{-1/8}\, ,
\end{equation}
with a suitable choice of the positive constant $C_{F, \Gamma, \gs}$. At this point we conclude with a Gaussian Large Deviation bound:
we are going  the exponential Chebychev inequality to the family of IID random variables 
$\{ \sup_{s\in [0, n^{-1/4}]} \in  \vert B_{ i, s_j+s}- B_{ i, s_j+s}\vert\}_{i=1, \ldots, n}$ 
whose law coincides with the law of $n^{-1/8}\sup_{t\in [0,1]} \vert B_t\vert$. 
Doob's submartingale inequality implies that $\bE \exp( a \sup_{t\in [0,1]} \vert B_t\vert) \le 2e \exp(a^2/2)$ for every $a\ge 0$
and therefore 
we have
\begin{equation}
\label{eq:aftobeE}
\sup_{j=0, \ldots, \ell_n } \bbE \bP \left( \frac 1n \sum_{i=1}^n \sup_{s\in [0, n^{-1/4}]}  \left \vert  X_{i,s_j+s}- X_{i,s_j} \right \vert >\gd
\right)\,\le \, \exp\left( -c {\gd^2} {n^{1/4}}\right)\,=\, o \left(1/\ell_n\right) \, ,
\end{equation}
for some $c>0$.
But the left-hand side of \eqref{eq:aftobeE} is just a rewriting of the left-hand side of
\eqref{eq:probabest}, so we are done.
 
\end{proof}

\section*{Acknowledgments}
G.G. is grateful to Bastien Fernandez, Roberto Livi  and Justin Salez for very helpful  discussions. We would like to thank the referees for their careful reading and useful remarks on the paper.

\end{document}